\documentclass[journal]{IEEEtran}

\usepackage[noadjust]{cite}
\usepackage{amsmath,amssymb,amsfonts,mathtools}
\usepackage{algorithmic}
\usepackage{arydshln}
\usepackage{graphicx}
\usepackage{bm}

\usepackage[inline,shortlabels]{enumitem}

\usepackage{amsthm}

\newcommand*{\suchthat}{%
    \;%
    \ifnum\currentgrouptype=16\relax
        \middle|%
    \else%
        \iftoggle{WithinBracMacro}{
            |
        }{
            |%
        }%
    \fi%
    \;%
}%
\makeatother

\newtheorem{definition}{Definition}

\newtheorem{remark}{Remark}
\newtheorem{theorem}{Theorem}
\newtheorem{lemma}{Lemma}

\theoremstyle{definition}

\newenvironment{smallarray}[1]
 {\null\,\vcenter\bgroup\scriptsize
  \arraycolsep=.13885em
  \hbox\bgroup$\array{@{}#1@{}}}
 {\endarray$\egroup\egroup\,\null}

\DeclareRobustCommand{\svdots}{
  \vbox{%
    \baselineskip=0.33333\normalbaselineskip
    \lineskiplimit=0pt
    \hbox{.}\hbox{.}\hbox{.}%
    \kern-0.2\baselineskip
  }%
}

\begin{document}

\title{Efficient Lyapunov-Based Stabilizability and Detectability Tests: From LTI to LPV Systems}
\author{T.~J.~{Meijer}, V.~S.~{Dolk} and W.~P.~M.~H.~{Heemels},
\thanks{Corresponding author: T.~J.~{Meijer}.}
\thanks{Tomas Meijer and Maurice Heemels are with the Department of Mechanical Engineering, Eindhoven University of Technology, The Netherlands. (e-mail: t.j.meijer@tue.nl; m.heemels@tue.nl).}%
\thanks{Victor Dolk is with ASML, De Run 6665, 5504 DT Veldhoven, The Netherlands. (e-mail: victor.dolk@asml.com)}}

\maketitle

\begin{abstract}
	In this technical note, we generalize the well-known Lyapunov-based stabilizability and detectability tests for linear time-invariant (LTI) systems to the context of discrete-time (DT) polytopic linear parameter-varying (LPV) systems. To do so, we exploit the mathematical structure of the class of poly-quadratic Lyapunov functions, which enables us to formulate conditions in the form of linear matrix inequalities (LMIs). Our results differ from existing conditions in that we adopt weaker assumptions on the parameter dependence of the controllers/observers and our method does not require explicitly incorporating these gains, which renders the new conditions less computationally demanding. Interestingly, our results also have important implications for existing controller and observer synthesis techniques based on poly-QLFs. In fact, we show that existing observer synthesis results are stronger than was anticipated in the sense that they are necessary for a larger class of observers. Moreover, we also introduce new controller synthesis conditions and illustrate our results by means of a numerical case study.
\end{abstract}


\section{Introduction}
\IEEEPARstart{S}{ystematic} construction of Lyapunov functions (LFs) is a crucial problem in many control applications. By constructing a LF for a (nonlinear) system, we can check and certify its (robust) stability properties. By simultaneously designing a controller or estimator while constructing LFs for the resulting closed-loop or estimation error dynamics, we can synthesize stabilizing controllers and observers. In addition, we can incorporate and optimize closed-loop performance criteria in the design procedure~\cite{Scherer2011}. What makes these techniques even more powerful are the converse Lyapunov theorems, see, e.g.,~\cite{Khalil1996-2nd-edition,Jiang2002}, which show that, generally, any stable system admits a (converse) LF. In practice, however, (converse) LFs often cannot be systematically constructed. Linear time-invariant (LTI) systems form an exception to this because they are asymptotically stable if and only if they admit a quadratic Lyapunov function (QLF)~\cite{Hespanha2018}, which can be constructed using LMI-based conditions~\cite{Scherer2011}. This fact and the non-conservatism it induces, is exploited in many powerful LMI-based tools for LTI systems, such as the well-known Lyapunov-based stabilizability and detectability analysis conditions, see, e.g.,~\cite{Hespanha2018}, for DT LTI systems of the form
\begin{equation}
    \begin{aligned}
        x_{k+1} &= Ax_k + Bu_k,\\
        y_k &= Cx_k,
    \end{aligned}
    \label{eq:lti-system}
\end{equation}
where $x_k\in\mathbb{R}^{n_x}$, $u_k\in\mathbb{R}^{n_u}$ and $y_k\in\mathbb{R}^{n_y}$ denote the state, input and output, respectively, and $A$, $B$ and $C$ are matrices of appropriate dimensions. According to these tests, the system~\eqref{eq:lti-system} is detectable if and only if there exists a symmetric-positive-definite matrix $\bar{P}\in\mathbb{S}^{n_x}_{\succ 0}$ such that~\cite{Hespanha2018}
\begin{equation}
    \bar{P} - A^\top \bar{P}A + C^\top C\succ 0.
    \label{eq:lti-det-lmi}
\end{equation}
Similarly, the system~\eqref{eq:system} is stabilizable if and only if there exists $\bar{S}\in\mathbb{S}^{n_x}_{\succ 0}$ satisfying~\cite{Hespanha2018}
\begin{equation}
    \bar{S} - A\bar{S}A^\top +BB^\top \succ 0.
    \label{eq:lti-stab-lmi}
\end{equation}
Interestingly, the conditions above do not compute an observer gain or a controller gain that renders the resulting, respectively, error system and closed-loop system asymptotically stable (however, we can reconstruct such gains based on the obtained solutions $\bar{P}$ and $\bar{S}$, respectively). Consequently,~\eqref{eq:lti-det-lmi} and~\eqref{eq:lti-stab-lmi} are more computationally efficient than their synthesis counterparts that do compute these gains explicitly, next to the LF.

In this note, we are interested in a class of discrete-time polytopic LPV systems, see~\eqref{eq:system} below, for which the existence of a QLF is only sufficient and, thereby, conservative~\cite{Mason2007} for global uniform asymptotic stability (GUAS). With the aim of striking a better balance between ease of constructibility and degree of conservatism, poly-quadratic Lyapunov functions (poly-QLFs) have been adopted in the literature. By exploiting the polytopic structure of this class of LFs and the system itself, poly-QLFs can be constructed using LMI-based techniques. This idea has resulted in the development of many powerful analysis and synthesis techniques, see, e.g.,~\cite{Daafouz2001,Heemels2010,Oliveira2022,Pereira2021}. Many of these conditions are not only sufficient, but also necessary for the existence of a poly-QLF and, possibly, a controller or observer with a certain polytopic structure, see, e.g.,~\cite{Pandey2018,Heemels2010,Daafouz2001}. In other words, the systematic, powerful and non-conservative nature of many Lyapunov-based techniques that we have come to love in the LTI setting, can be and have been extended for this class of LPV systems. However, to the best of our knowledge, no generalizations of the Lyapunov-based tests in~\eqref{eq:lti-det-lmi} and~\eqref{eq:lti-stab-lmi} to LPV systems have been proposed yet and this is one of the gaps we aim to fill in this note.

More generally speaking, in this work, we consider stabilizing controller and observer synthesis for polytopic LPV systems (hence, including stabilizability and detectability tests), for which we propose several novel LMI-based analysis and synthesis conditions. In doing so, this note also performs a brief literature survey by putting several existing LPV synthesis conditions into perspective, relative to which we position our own contributions: Firstly, we propose new necessary and sufficient LMI-based \textit{stability} conditions for LPV systems (\ref{item:lem-equiv-pqs-4} in Lemma~\ref{lem:equiv-pqs} below) that, analogous to~\cite{Daafouz2001}, can be used to synthesize parameter-dependent poly-quadratically (poly-Q) stabilizing, i.e., rendering the closed-loop system poly-QS, controllers. In contract with the conditions in~\cite{Daafouz2001}, we do not introduce additional slack variables, which results in lower computational complexity for the corresponding LMIs. Secondly, we also develop a novel LMI-based analysis condition that is \textit{necessary and sufficient} for poly-Q \textit{detectability} (Theorem~\ref{thm:poly-q-detectability-analysis} below), which essentially generalizes~\eqref{eq:lti-det-lmi} from the LTI case to the polytopic LPV case. Interestingly, we are able to obtain observers with the same structure as in~\cite{Heemels2010} without restricting the parameter-dependence a-priori as done in~\cite{Heemels2010}. Thereby, we show that the results in~\cite{Heemels2010} are even stronger than claimed originally in their initial publication as they are not only necessary for the existence of a particular observer, but even necessary for a more generic structure, see Remark~\ref{rem:heemels2010}. Thirdly, we also generalize~\eqref{eq:lti-stab-lmi} from the LTI to the polytopic LPV setting by developing, separately, a parameter-dependent necessary and sufficient analysis condition (Theorem~\ref{thm:poly-quadratic-stabilizability-nec-suff}), however, the parameter dependence makes it intractable to systematically verify its feasibility. Therefore, we also provide sufficient conditions (Theorem~\ref{thm:poly-q-stabilizability-analysis}) in the form of a finite number of LMIs that can be systematically solved. Interestingly, note that our novel conditions are equally valid in the context of \textit{switched linear systems}~\cite{Philippe2016,Athanasopoulos2014,Liberzon2003} as discussed in~\cite{Mason2007}, which makes their use even broader. We demonstrate our results in a numerical case study and investigate their computational complexity (see Section~\ref{sec:comp-compl}).

The remainder of this note is organized as follows. In Section~\ref{sec:prelim}, we introduce the model class and relevant preliminaries. Sections~\ref{sec:detectability} and~\ref{sec:stabilizability} present our main results on detectability and stabilizability. We illustrate our results in a numerical case study in Section~\ref{sec:case-study}. Section~\ref{sec:conclusions} presents some conclusions and all proofs can be found in the Appendix.

{\bf Notation.} The sets of real and non-negative real and natural numbers are denoted, respectively, $\mathbb{R}=(-\infty,\infty)$, $\mathbb{R}_{\geqslant 0}=[0,\infty)$ and $\mathbb{N}=\{0,1,2,\hdots\}$. We also denote $\mathbb{N}_{\geqslant n}=\{n,n+1,n+2,\hdots\}$ and $\mathbb{N}_{[n,m]}=\{n,n+1,n+2,\hdots,m\}$ for $n,m\in\mathbb{N}$. The set of $n$-by-$n$ symmetric matrices is denoted by $\mathbb{S}^n$. $M\preccurlyeq 0$, $M\succ 0$ and $M\succcurlyeq 0$ mean, respectively, that $M\in\mathbb{S}^n$ is negative semi-definite, positive definite and positive semi-definite. The sets of such matrices are $\mathbb{S}^{n}_{\preccurlyeq 0}$, $\mathbb{S}^n_{\succ 0}$ and $\mathbb{S}^{n}_{\succcurlyeq 0}$, respectively. The symbol $I$ is an identity matrix of appropriate dimensions and $\star$ completes a symmetric matrix, e.g., $\begin{bmatrix}\begin{smallmatrix} A & \star\\B & C\end{smallmatrix}\end{bmatrix}=\begin{bmatrix}\begin{smallmatrix}A & B^\top\\B & C\end{smallmatrix}\end{bmatrix}$. For $A\in\mathbb{R}^{n\times m}$, $A^\top$ is its transpose, $\operatorname{He}(A)=A+A^\top$ and $A_{\perp}$ is a matrix whose columns are a basis for $\operatorname{ker}A$. Moreover, $\operatorname{diag}\{A_i\}_{i\in\mathbb{N}_{[1,n]}}$ is a block-diagonal matrix with blocks $A_i$, $i\in\mathbb{N}_{[1,n]}$. Let $\bm{e}_i = (0_{(i-1)\times 1},1,0_{(n-i)\times 1})\in\mathbb{R}^{n}$, $i\in\mathbb{N}_{[1,n]}$, be the $i$-th unit vector. The notation $x\leftarrow y$ means ``replace'' $x$ by $y$.

\section{Model class and preliminaries}\label{sec:prelim}
We consider discrete-time LPV systems of the form
\begin{equation}
    \begin{aligned}
		x_{k+1} &= A(p_k)x_k + Bu_k,\\
        y_k &= Cx_k,	
    \end{aligned}
	\label{eq:system}
\end{equation}
where $x_k\in\mathbb{R}^{n_x}$, $u_k\in\mathbb{R}^{n_u}$ and $y_k\in\mathbb{R}^{n_y}$ denote, respectively, the state, control input and measured output at time $k\in\mathbb{N}$. The unknown parameter $p_k$ belongs to a known compact set $\mathbb{P}$, i.e., $p_k\in\mathbb{P}\subseteq\mathbb{R}^{n_p}$ for all $k\in\mathbb{N}$. 
\begin{definition}
	The set of admissible parameter sequences, denoted $\mathcal{P}$, is the set of all parameter sequences $\{p_k\}_{k\in\mathbb{N}}$ with $p_k\in\mathbb{P}$ for all $k\in\mathbb{N}$.
\end{definition}
\noindent The matrix-valued function $A\colon\mathbb{P}\rightarrow\mathbb{R}^{n_x\times n_x}$ and matrices $B\in\mathbb{R}^{n_x\times n_u}$ and $C\in\mathbb{R}^{n_y\times n_x}$ are known. We focus on polytopic systems~\eqref{eq:system}, whose specific mathematical structure often facilitates the development of LMI-based analysis and synthesis conditions, see, e.g.~\cite{Daafouz2001,Heemels2010,Oliveira2022,Pereira2021}.
\begin{definition}\label{dfn:polytopic}
	The system~\eqref{eq:system} is said to be polytopic, if there exist continuous\footnote{Continuity of the functions $\xi_i$, $i\in\mathcal{N}$, is assumed so that we can apply Lemma~\ref{lem:finsler} in the Appendix. However, the same result can be obtained using the fact that $A$ and $P$, introduced later in~\eqref{eq:poly-lyap}, are uniformly bounded on $\mathbb{P}$.} functions $\xi_i\colon\mathbb{P}\rightarrow\mathbb{R}_{\geqslant 0}$, $i\in\mathcal{N}\coloneqq\mathbb{N}_{[1,N]}$, such that the mapping $\xi\coloneqq(\xi_1,\xi_2,\hdots,\xi_N)$ satisfies $\xi(\mathbb{P})\subset\mathbb{X}\coloneqq\{\mu=(\mu_1,\mu_2,\hdots,\mu_N)\in\mathbb{R}_{\geqslant 0}^N\mid\sum_{i\in\mathcal{N}}\mu_i=1\}$, and matrices $A_i\in\mathbb{R}^{n_x\times n_x}$, $i\in\mathcal{N}$, such that, for all $\pi\in\mathbb{P}$,
	\begin{equation}
		A(\pi) = \sum_{\mathclap{i\in\mathcal{N}}}\xi_i(\pi)A_i.
		\label{eq:poly-A}
	\end{equation}
	If, in addition, the functions $\xi_i$, $i\in\mathcal{N}$, satisfy $\{\bm{e}_i\}_{i\in\mathcal{N}}\subset\xi(\mathbb{P})$, the system~\eqref{eq:system} is strictly polytopic.
\end{definition}
\noindent Systems~\eqref{eq:system} with affine parameter dependence and a polytopic parameter set admit a strictly polytopic representation. Other systems can often be described by polytopic systems~\eqref{eq:system}.

\subsection{Global uniform asymptotic and poly-quadratic stability}
Let $x_k(\bm{p},\bm{u},x_0)$ be the solution to~\eqref{eq:system} at time $k\in\mathbb{N}$ for initial state $x_0\in\mathbb{R}^{n_x}$, parameter sequence $\bm{p}\in\mathcal{P}$ and input sequence $\bm{u}=\{u_j\}_{j\in\mathbb{N}}$ with $u_j\in\mathbb{R}^{n_u}$ for all $j\in\mathbb{N}$. For systems~\eqref{eq:system} without input, i.e., $\bm{u}=\bm{0}=\{0\}_{j\in\mathbb{N}}$, We define GUAS~\cite{Khalil1996-2nd-edition}, below, where we consider uniformity with respect to parameter sequence $\bm{p}$ (and initial state $x_0$).
\begin{definition}\label{dfn:guas}
	The system~\eqref{eq:system} is said to be GUAS, if
	\begin{enumerate}[labelindent=0pt,labelwidth=\widthof{\ref{item:def-guas-2}},label=(D\ref{dfn:guas}.\arabic*),itemindent=0em,leftmargin=!]
		\item \label{item:def-guas-1} for any positive $\epsilon\in\mathbb{R}_{>0}$, there exists a positive $\delta(\epsilon)\in\mathbb{R}_{>0}$ (independent of $\bm{p}$) such that
		\begin{equation}
			\|x_0\|<\delta(\epsilon)\implies \|x_k(\bm{p},\bm{0},x_0)\|<\epsilon,
		\end{equation}
		for all $k\in\mathbb{N}_{\geqslant 1}$ and for all $\bm{p}\in\mathcal{P}$, and
		\item \label{item:def-guas-2} for any positive $\rho,\epsilon\in\mathbb{R}_{>0}$, there exists $K(\rho,\epsilon)\in\mathbb{N}$ (independent of $\bm{p}$) such that, for all $k>K(\rho,\epsilon)$ and for all $\bm{p}\in\mathcal{P}$,
		\begin{equation}
			\|x_0\|<\rho \implies \|x_k(\bm{p},\bm{0},x_0)\|<\epsilon.
		\end{equation}
	\end{enumerate}
\end{definition}\vspace*{0\topsep}
\noindent For guaranteeing GUAS of~\eqref{eq:system} it is often practical to use Lyapunov-based techniques. For polytopic systems~\eqref{eq:system}, poly-QLFs, as introduced below, have been used extensively, see, e.g.,~\cite{Daafouz2001,Heemels2010,Oliveira2022,Pereira2021}, because they \begin{enumerate*}[label=(\alph*)] \item can be constructed using LMI-based techniques, and \item are less conservative than traditional QLFs \end{enumerate*}. These poly-QLFs lead us to define poly-quadratic stability (poly-QS) below.
\begin{definition}\label{dfn:polyQS}
	The system~\eqref{eq:system} being polytopic is said to be poly-QS, if it admits a poly-QLF, i.e., a function $V\colon\mathbb{P}\times\mathbb{R}^{n_x}\rightarrow\mathbb{R}_{\geqslant 0}$ of the form
	\begin{equation}
		V(\pi,x) = x^\top P(\pi)x\text{ with }P(\pi) = \sum_{i\in\mathcal{N}}\xi_i(\pi)\bar{P}_i,
		\label{eq:poly-lyap}
	\end{equation}
	where $\bar{P}_i\in\mathbb{S}^{n_x}$ and $\xi_i$, $i\in\mathcal{N}$, are the functions in~\eqref{eq:poly-A} satisfying, for some $a_j\in\mathbb{R}_{>0}$, $j\in\mathbb{N}_{[1,3]}$,
	\begin{align}
		a_1\|x\|^2&\leqslant V(\pi,x)\leqslant a_2\|x\|^2,\label{eq:lyap-bounds}\\
		V(\pi_+,x_+)&\leqslant V(\pi,x)-a_3\|x\|^2,\label{eq:lyap-desc}
	\end{align}
	for all $\pi_+,\pi\in\mathbb{P}$ and $x_+,x\in\mathbb{R}^{n_x}$ with $x_+=A(\pi)x$.
\end{definition}
\noindent Poly-QS implies GUAS~\cite{Khalil1996-2nd-edition}, however, the converse is in general not true and, therefore, some degree of conservatism is introduced. The use of poly-QLFs is strongly related to the use of mode-dependent LFs for switched linear systems~\cite{Mason2007}.

\subsection{Relevant and novel LMI-based poly-QS conditions}
Several necessary and sufficient LMI-based conditions, which are summarized below, have been proposed in the literature to verify poly-QS, see, e.g.,~\cite{Daafouz2001,Heemels2010,Pandey2018}. We also propose a \textit{new} condition,~\ref{item:lem-equiv-pqs-4} below, with fewer decision variables and, thereby, lower computational complexity.
\begin{lemma}\label{lem:equiv-pqs}
	Consider the system~\eqref{eq:system} being polytopic and the statements:
	\begin{enumerate}[labelindent=0pt,labelwidth=\widthof{\ref{item:lem-equiv-pqs-4}},label=(L\ref{lem:equiv-pqs}.\arabic*),itemindent=0em,leftmargin=!]
		\item \label{item:lem-equiv-pqs-1} System~\eqref{eq:system} is poly-QS with~\eqref{eq:poly-lyap} being a poly-QLF.
		\item \label{item:lem-equiv-pqs-2} There exist matrices $\bar{P}_i\in\mathbb{S}^{n_x}$, $\bar{S}_i=\bar{P}_i^{-1}$ and $X_i\in\mathbb{R}^{n_x\times n_x}$, $i\in\mathcal{N}$, such that
	\begin{equation}
		\begin{bmatrix}\begin{smallmatrix}
			X_i + X_i^\top - \bar{S}_i & \star\\
			A_iX_i & \bar{S}_j
		\end{smallmatrix}\end{bmatrix}\succ 0,\text{ for all }i,j\in\mathcal{N}.
		\label{eq:lem-equiv-pqs-2}
	\end{equation}
		\item \label{item:lem-equiv-pqs-3} There exist matrices $\bar{P}_i\in\mathbb{S}^{n_x}$ and $X_i\in\mathbb{R}^{n_x\times n_x}$, $i\in\mathcal{N}$, such that
	\begin{equation}
		\begin{bmatrix}\begin{smallmatrix}
			X_i + X_i^\top - \bar{P}_j & \star\\
			A_i^\top X_i^\top & \bar{P}_i
		\end{smallmatrix}\end{bmatrix}\succ 0,\text{ for all }i,j\in\mathcal{N}.
		\label{eq:lem-equiv-pqs-3}
	\end{equation}
	\item \label{item:lem-equiv-pqs-4} There exist matrices $\bar{P}_i\in\mathbb{S}^{n_x}$ and $\bar{S}_i=\bar{P}_i^{-1}$, $i\in\mathcal{N}$, such that
	\begin{equation}	
		\begin{bmatrix}\begin{smallmatrix}
			\bar{S}_i & \star\\
			A_i\bar{S}_i & \bar{S}_j
		\end{smallmatrix}\end{bmatrix}\succ 0,\text{ for all }i,j\in\mathcal{N}.
		\label{eq:lem-equiv-pqs-4}
	\end{equation}
	\end{enumerate}
	It holds that $\ref{item:lem-equiv-pqs-4}\Leftrightarrow\ref{item:lem-equiv-pqs-2}\Rightarrow\ref{item:lem-equiv-pqs-1}$ and~$\ref{item:lem-equiv-pqs-3}\Rightarrow\ref{item:lem-equiv-pqs-1}$. The matrices $X_i$, $i\in\mathcal{N}$, satisfying~\eqref{eq:lem-equiv-pqs-2} or~\eqref{eq:lem-equiv-pqs-3} are non-singular. In addition, if system~\eqref{eq:system} is strictly polytopic, then all statements are equivalent.
\end{lemma}
\noindent While mathematically equivalent (for strictly polytopic systems), the statement above have different applications, i.e., \ref{item:lem-equiv-pqs-2} and~\ref{item:lem-equiv-pqs-4} are useful to obtain LMI-based controller synthesis conditions, whereas~\ref{item:lem-equiv-pqs-3} is used for observer design, see~\cite{Daafouz2001,Heemels2010,Pandey2018}. Our \textit{novel} condition~\ref{item:lem-equiv-pqs-4} leads to controller synthesis conditions with reduced computational complexity, as we discuss in Section~\ref{sec:stabilizability-synth}.  

\section{Observer design and detectability}\label{sec:detectability}
In this section, we consider the problem of designing observers of the form
\begin{equation}
	\hat{x}_{k+1} = A(p_k)\hat{x}_k + L(k,\bm{p})C(\hat{x}_k-x_k) + Bu_k,
	\label{eq:observer}
\end{equation}
where $\hat{x}_k\in\mathbb{R}^{n_x}$ denotes the estimated state at time $k\in\mathbb{N}$ and $L\colon\mathbb{N}\times\mathcal{P}\rightarrow\mathbb{R}^{n_x\times n_y}$ is the to-be-designed observer gain. For the error $e_k\coloneqq \hat{x}_k-x_k$, $k\in\mathbb{N}$, we obtain the estimation-error system for~\eqref{eq:system} and~\eqref{eq:observer} that is linear in $e$, i.e.,
\begin{equation}
	e_{k+1} = (A(p_k)+ L(k,\bm{p})C)e_k.
	\label{eq:error-system}
\end{equation}
We are interested in analyzing--in a systematic fashion--whether a given system is detectable in the sense that an observer of the form~\eqref{eq:observer} exists, which is such that the corresponding error system~\eqref{eq:error-system} is GUAS. The class of observers in~\eqref{eq:observer} is more general than the class of observers considered in, e.g.,~\cite{Heemels2010} in that it depends on the entire parameter sequence $\bm{p}$ (rather than $p_k$ alone). This allows us to exploit a history or prediction of $\bm{p}$, if available, so that we can study whether this enables us to design observers, that render~\eqref{eq:error-system} GUAS, for systems for which we are not able to do so using $p_k$ alone. Also, we use the term poly-Q detectability to indicate that an observer~\eqref{eq:observer} exists that renders~\eqref{eq:error-system} poly-QS, as formalized below. Some degree of conservatism is introduced by considering poly-Q detectability instead of, for instance, a more general detectability definition in terms of GUAS.
\begin{definition}\label{dfn:poly-quadratic-detectability}
	The system~\eqref{eq:system} is said to be poly-Q detectable, if there exists a function $L\colon\mathbb{N}\times\mathcal{P}\rightarrow\mathbb{R}^{n_x\times n_y}$ that renders the error system~\eqref{eq:error-system} poly-QS.
\end{definition}\vspace*{-\topsep}

\subsection{Poly-Q-detectability analysis}
Next, we propose novel LMI-based conditions that generalize the Lyapunov-based detectability test~\eqref{eq:lti-det-lmi} for LTI systems.
\begin{theorem}\label{thm:poly-q-detectability-analysis}
	The system~\eqref{eq:system} being strictly polytopic is poly-Q detectable, if and only if there exist matrices $\bar{P}_i\in\mathbb{S}^{n_x}_{\succ 0}$, $i\in\mathcal{N}$, such that
	\begin{equation}
		\bar{P}_i-A_i^\top \bar{P}_jA_i+C^\top C\succ 0,\text{ for all }i,j\in\mathcal{N}.
		\label{eq:poly-q-detectability-analysis}
	\end{equation}
	If the system~\eqref{eq:system} is polytopic, the conditions are only sufficient. Moreover, in that case,~\eqref{eq:observer} with, for $k\in\mathbb{N}$ and $\bm{p}\in\mathcal{P}$,
	\begin{equation}
		L(k,\bm{p}) = -\sum_{i\in\mathcal{N}}\xi_i(p_k) A_i(\bar{P}_i+C^\top C)^{-1}C^\top
		\label{eq:poly-q-detectability-analysis-gain}
	\end{equation}	 
	renders the error system~\eqref{eq:error-system} poly-QS with poly-QLF~\eqref{eq:poly-lyap}.
\end{theorem}
\noindent The conditions above provide a non-conservative test for poly-Q detectability of strictly polytopic systems~\eqref{eq:system}. Observe also that Theorem~\ref{thm:poly-q-detectability-analysis} is more general than the results in~\cite{Heemels2010,Pandey2018} in the sense that--unlike~\cite{Heemels2010,Pandey2018}--we do not require $L$ to depend only on the current parameter $p_k$ and be polytopic, but more importantly Theorem~\ref{thm:poly-q-detectability-analysis} is simpler and computationally more efficient than the observer synthesis conditions in, e.g.,~\cite[Theorem 2]{Heemels2010,Pandey2018} since the observer gain $L$ does not appear explicitly in~\eqref{eq:poly-q-detectability-analysis} and no slack variables were introduced.

Once we have obtained a feasible solution to~\eqref{eq:poly-q-detectability-analysis}, we can use it to construct the observer gain in~\eqref{eq:poly-q-detectability-analysis-gain} that renders~\eqref{eq:error-system} poly-QS. Interestingly, despite the fact that we do not impose this a-priori, $L$ in~\eqref{eq:poly-q-detectability-analysis-gain} is still polytopic and depends only on $p_k$. At first glance, readers may find the uncovered polytopic structure for $L$ rather logical, however, we will see that this is a surprising result in the next section, where we obtain a very different structure in the context of stabilizing feedback control. In the LTI context, we can use a solution to the LTI detectability test~\eqref{eq:lti-det-lmi} to construct (using the Woodbury identity~\cite{Golub1996}) the stabilizing observer gain~\cite{Duan2013}
\begin{equation}
    \bar{L} = -\bar{A}(\bar{P}+C^\top C)^{-1}C^\top = -\bar{A}\bar{S}C(I+C\bar{S}C^\top)^{-1},
    \label{eq:lti-det-gain}
\end{equation}
with $\bar{S}$ satisfying~\eqref{eq:lti-det-lmi}. In fact, if we substitute $P_i=\bar{P}\in\mathbb{S}_{\succ 0}^{n_x}$ and $A_i\in\mathbb{R}^{n_x\times n_x}$, $i\in\mathcal{N}$, in~\eqref{eq:poly-q-detectability-analysis} and~\eqref{eq:poly-q-detectability-analysis-gain}, we recover~\eqref{eq:lti-det-lmi} and~\eqref{eq:lti-det-gain}, respectively, which shows that Theorem~\ref{thm:poly-q-detectability-analysis} truly generalizes the well-known Lyapunov stabilizability results from the LTI setting to polytopic LPV systems~\eqref{eq:system}.

\begin{remark}\label{rem:heemels2010}
The fact that the conditions in Theorem~\ref{thm:poly-q-detectability-analysis} yield $L$ as in~\eqref{eq:poly-q-detectability-analysis}, i.e., polytopic and only dependent on $p_k$, also means that the conditions in~\cite{Heemels2010,Pandey2018} are necessary and sufficient for the existence of a more general $L$ as in~\eqref{eq:observer} that renders~\eqref{eq:error-system} poly-QS. In other words, the system~\eqref{eq:system} being (strictly) polytopic is poly-Q detectable, if (and only if) there exist matrices $\bar{P}_i\in\mathbb{S}^{n_x}$, $X_i\in\mathbb{R}^{n_x\times n_x}$ and $Y_i\in\mathbb{R}^{n_x\times n_y}$, $i\in\mathcal{N}$, such that
\begin{equation}
	\begin{bmatrix}\begin{smallmatrix}
		X_i+X_i^\top - \bar{P}_j & \star\\
		A_i^\top X_i^\top + C^\top Y_i^\top & \bar{P}_i
	\end{smallmatrix}\end{bmatrix}\succ 0,\text{ for all }i,j\in\mathcal{N}.
	\label{eq:cor}
\end{equation}
Moreover, in that case,~\eqref{eq:observer} with
\begin{equation}
	L(k,\bm{p}) = \sum_{i\in\mathcal{N}}\xi_i(p_k)X_i^{-1}Y_i,\quad k\in\mathbb{N}, \bm{p}\in\mathcal{P},
	\label{eq:cor-L}
\end{equation}
renders the error system~\eqref{eq:error-system} poly-QS with poly-QLF~\eqref{eq:poly-lyap}. The proof is omitted, but sufficiency is shown in~\cite{Pandey2018} and necessity (under strict polytopicity) follows from Theorem~\ref{thm:stab-synth} and the fact that the conditions in~\eqref{eq:cor} are necessary for the existence of a stabilizing observer~\eqref{eq:observer} with polytopic $L(\bm{p},k)=\bar{L}(p_k)$~\cite{Pandey2018}.
\end{remark}

\section{Controller design and stabilizability}\label{sec:stabilizability}
Next, we aim to design stabilizing state-feedback laws
\begin{equation}
	u_k = K(k,\bm{p})x_k\text{ for some }K\colon\mathbb{N}\times\mathcal{P}\rightarrow\mathbb{R}^{n_u\times n_x},
	\label{eq:fb-law}
\end{equation}
which are more general feedback laws than considered in~\cite{Daafouz2001} in the sense that $K$ is not necessarily polytopic and may depend on the entire parameter sequence $\bm{p}$. By considering~\eqref{eq:fb-law}, we ensure that the corresponding closed-loop system
\begin{equation}
	x_{k+1} = (A(p_k) + BK(k,\bm{p}))x_k,
	\label{eq:closed-loop}
\end{equation}
is--like~\eqref{eq:fb-law}--also linear in $x$. Since there may exist other poly-Q stabilizing controllers, focusing only on~\eqref{eq:fb-law} may come at the cost of being conservative. We are interested in systematically analyzing whether a given system~\eqref{eq:system} is stabilizable in the sense that a controller of the form~\eqref{eq:fb-law} exists that renders the closed-loop system~\eqref{eq:closed-loop} poly-QS, for which we will develop LMI-based analysis and synthesis conditions.
\begin{definition}\label{dfn:poly-q-stabilizability}
	The system~\eqref{eq:system} is said to be poly-Q stabilizable, if there exists a function $K\colon\mathbb{N}\times\mathcal{P}\rightarrow\mathbb{R}^{n_u\times n_x}$ that renders the closed-loop system~\eqref{eq:closed-loop} poly-QS.
\end{definition}\vspace*{-\topsep}

\subsection{Poly-quadratic-stabilizability analysis}
Inspired by Theorem~\ref{thm:poly-q-detectability-analysis}, which generalizes~\eqref{eq:lti-det-lmi} to polytopic LPV systems, we aim to generalize the Lyapunov-based stabilizability condition~\eqref{eq:lti-stab-lmi}. To this end, we first state parameter-dependent necessary and sufficient conditions below.
\begin{theorem}\label{thm:poly-quadratic-stabilizability-nec-suff}
    The system~\eqref{eq:system} being polytopic is poly-Q stabilizable, if and only if there exist matrices $\bar{S}_i\in\mathbb{S}^{n_x}_{\succ 0}$ and $\bar{P}_i=\bar{S}_i^{-1}$, $i\in\mathcal{N}$, such that $S(\pi)=P^{-1}(\pi)$, with $P$ as in~\eqref{eq:poly-lyap}, satisfies, for all $\pi_+,\pi\in\mathbb{P}$
    \begin{equation}
        S(\pi_+)-A(\pi)S(\pi)A(\pi) + BB^\top\succ 0.
        \label{eq:poly-quadratic-stabilizability-nec-suff}
    \end{equation}
    Moreover, in that case, 
	\begin{equation}
		K(k,\bm{p}) = -B^\top(S(p_{k+1})+BB^\top)^{-1}A(p_k),
		\label{eq:poly-q-stabilizability-analysis-gain}
	\end{equation}
    for $k\in\mathbb{N}$ and $\bm{p}\in\mathcal{P}$, renders the closed-loop system~\eqref{eq:closed-loop} poly-QS with poly-QLF~\eqref{eq:poly-lyap}.
\end{theorem}
\noindent Solving the conditions in Theorem~\ref{thm:poly-quadratic-stabilizability-nec-suff} for all $\pi_+,\pi\in\mathbb{P}$ is not tractable, which is why we typically attempt to translate them to conditions in terms of the vertices. It follows trivially from Theorem~\ref{thm:poly-quadratic-stabilizability-nec-suff} that, for strictly polytopic systems~\eqref{eq:system}, the existence of $\bar{S}_i\in\mathbb{S}^{n_x}_{\succ 0}$, $i\in\mathcal{N}$, such that 
\begin{equation}
    \bar{S}_j-A_i\bar{S}_iA_i^\top + BB^\top \succ 0,\text{ for all }i,j\in\mathcal{N},
    \label{eq:poly-q-stabilizability-analysis-nec}
\end{equation}
is a \emph{necessary} condition for poly-Q stabilizability. Unfortunately,~\eqref{eq:poly-q-stabilizability-analysis-nec} is--to the best of our knowledge--\emph{not sufficient}, since we are unable to derive, from~\eqref{eq:poly-q-stabilizability-analysis-nec}, a condition in terms of $\bar{P}_i$ (instead of $\bar{S}_i$) that also does not contain any products between matrices that depend on the same index $i$ or $j$, which is crucial in proving sufficiency of~\eqref{eq:poly-q-stabilizability-analysis-nec}. By introducing additional slack variables in Theorem~\ref{thm:poly-quadratic-stabilizability-nec-suff}, however, we obtain the conditions below, which are \emph{sufficient}.
\begin{theorem}\label{thm:poly-q-stabilizability-analysis}
	The system~\eqref{eq:system} being polytopic is poly-Q stabilizable, if there exist matrices $\bar{S}_i\in\mathbb{S}^{n_x}$ and $X_{i}\in\mathbb{R}^{n_x\times n_x}$, $\bar{P}_i=\bar{S}_i^{-1}$, $i\in\mathcal{N}$, such that
	\begin{equation}
		\begin{bmatrix}\begin{smallmatrix}
			X_i+X_i^\top -A_i\bar{S}_iA_i^\top + BB^\top & \star\\
			X_i & \bar{S}_j
		\end{smallmatrix}\end{bmatrix}\succ 0,\text{ for all }i,j\in\mathcal{N}.
		\label{eq:poly-q-stabilizability-analysis}
	\end{equation}
	Moreover, in that case, $K$ in~\eqref{eq:poly-q-stabilizability-analysis-gain} with $P$ in~\eqref{eq:poly-lyap} and $S(\pi)=P^{-1}(\pi)$ renders the closed-loop system~\eqref{eq:closed-loop} poly-QS with poly-QLF~\eqref{eq:poly-lyap}.
\end{theorem}
\noindent Theorem~\ref{thm:poly-q-stabilizability-analysis} provides LMI-based conditions to guarantee that a given polytopic system~\eqref{eq:system} is poly-Q stabilizable. Observe that, in contrast with~\eqref{eq:poly-q-detectability-analysis-gain} in the context of poly-Q detectability, $K$ in~\eqref{eq:poly-q-stabilizability-analysis-gain} is not polytopic and depends on both $p_k$ and $p_{k+1}$. As mentioned before in Section~\ref{sec:detectability}, this emphasizes that the polytopic and $p_k$-dependent structure that we found for $L$ in~\eqref{eq:poly-q-detectability-analysis-gain} is not necessarily natural and/or obvious. 

By substituting $A(\pi) = \bar{A}\in\mathbb{R}^{n_x\times n_x}$ and $S(\pi)=\bar{S}=\bar{P}^{-1}\in\mathbb{S}^{n_x}_{\succ 0}$, $\pi\in\mathbb{P}$, in~\eqref{eq:poly-q-stabilizability-analysis-gain}, we recover (using the Woodbury identity~\cite{Golub1996})
\begin{equation}
	\bar{K} = -B^\top(\bar{S}+BB^\top)^{-1}\bar{A} = -(I+B^\top\bar{P}B)^{-1}\bar{P}\bar{A},
\end{equation}
with $\bar{P}=\bar{S}^{-1}\succ 0$ satisfying~\eqref{eq:lti-stab-lmi}, which is precisely the stabilizing controller gain obtained from~\eqref{eq:lti-stab-lmi} in the LTI setting, see, e.g.,~\cite{Duan2013}. Similarly, the conditions in Theorem~\ref{thm:poly-quadratic-stabilizability-nec-suff} reduce to those in~\eqref{eq:lti-stab-lmi} when making the same substitutions and, therefore, generalize~\eqref{eq:lti-stab-lmi} to the polytopic LPV setting.

\subsection{Poly-Q-stabilizing controller synthesis}\label{sec:stabilizability-synth}
In this section, we provide synthesis conditions, which differ from Theorem~\ref{thm:poly-q-stabilizability-analysis} in that they explicitly contain the controller gains, to design controllers of the form~\eqref{eq:fb-law}. 
\begin{theorem}\label{thm:stab-synth}
	Consider the system~\eqref{eq:system} being polytopic and the statements:
	\begin{enumerate}[labelindent=0pt,labelwidth=\widthof{\ref{item:thm-stab-synth-4}},label=(T\ref{thm:stab-synth}.\arabic*),itemindent=0em,leftmargin=!]
		\item \label{item:thm-stab-synth-1} System~\eqref{eq:system} is poly-Q stabilizable with $K(k,\bm{p})=\sum_{i,j\in\mathcal{N}}\xi_i(p_k)\xi_j(p_{k+1})\bar{K}_{ij}(\sum_{m\in\mathcal{N}}\xi_m(p_{k+1})X_{im})^{-1}$ for some $\bar{K}_{ij}\in\mathbb{R}^{n_u\times n_x}$, $X_{ij}\in\mathbb{R}^{n_x\times n_x}$, $i,j\in\mathcal{N}$.
		\item \label{item:thm-stab-synth-2} System~\eqref{eq:system} is poly-Q stabilizable with $K(k,\bm{p})=\sum_{i\in\mathcal{N}}\xi_i(p_k)\bar{K}_i$ for some $\bar{K}_i\in\mathbb{R}^{n_u\times n_x}$, $i\in\mathcal{N}$.
		\item \label{item:thm-stab-synth-3} There exist matrices $\bar{P}_i\in\mathbb{S}^{n_x}$, $\bar{S}_i=\bar{P}_i^{-1}$, $\bar{K}_i\in\mathbb{R}^{n_u\times n_x}$ and $Y_i=\bar{K}_i\bar{S}_i$, $i\in\mathcal{N}$, such that
	\begin{equation}	
		\begin{bmatrix}\begin{smallmatrix}
			\bar{S}_i & \star\\
			A_i\bar{S}_i+BY_i & \bar{S}_j
		\end{smallmatrix}\end{bmatrix}\succ 0,\text{ for all }i,j\in\mathcal{N}.
		\label{eq:thm-stab-synth-3}
	\end{equation}
		\item \label{item:thm-stab-synth-4} There exist matrices $\bar{P}_i\in\mathbb{S}^{n_x}$, $\bar{S}_i=\bar{P}_i^{-1}$, $X_{ij}\in\mathbb{R}^{n_x\times n_x}$, $\bar{K}_{ij}\in\mathbb{R}^{n_u\times n_x}$, $Y_{ij}=\bar{K}_{ij}X_{ij}$ and $Z_i\in\mathbb{R}^{n_x\times n_x}$, $i,j\in\mathcal{N}$, such that
		\begin{equation}
			\begin{bmatrix}\begin{smallmatrix}
				X_{ij}+X_{ij}^\top - \bar{S}_i & \star & \star\\
				A_iX_{ij} + BY_{ij} & Z_i+Z_i^\top & \star\\
				0 & Z_i & \bar{S}_j
			\end{smallmatrix}\end{bmatrix}\succ 0,\text{ for all }i,j\in\mathcal{N}.
			\label{eq:thm-stab-synth-4}
		\end{equation}
	\end{enumerate}
	Then, $\ref{item:thm-stab-synth-2}\Rightarrow\ref{item:thm-stab-synth-1}$ with $\bar{K}_{ij}=\bar{K}_i$ and $X_{ij}=I$, $i,j\in\mathcal{N}$. Moreover, $\ref{item:thm-stab-synth-3}\Rightarrow\ref{item:thm-stab-synth-2}$ with the same $\bar{K}_i$, $i\in\mathcal{N}$, and $\ref{item:thm-stab-synth-4}\Rightarrow\ref{item:thm-stab-synth-1}$ with the same $\bar{K}_{ij}$ and $X_{ij}$, $i,j\in\mathcal{N}$. If~\ref{item:thm-stab-synth-3} or~\ref{item:thm-stab-synth-4} holds,~\eqref{eq:poly-lyap} is a poly-QLF for the closed-loop system~\eqref{eq:closed-loop}. The matrices $X_{ij}$, $i,j\in\mathcal{N}$, satisfying~\eqref{eq:thm-stab-synth-4} are such that $X_i(\pi)=\sum_{j\in\mathcal{N}}\xi_j(\pi)X_{ij}$, $i\in\mathcal{N}$, are non-singular for all $\pi\in\mathbb{P}$. If~\eqref{eq:system} is also strictly polytopic, then $\ref{item:thm-stab-synth-2}\Leftrightarrow\ref{item:thm-stab-synth-3}$.
\end{theorem}
\noindent Statements~\ref{item:thm-stab-synth-3} and~\ref{item:thm-stab-synth-4} are LMIs in $\bar{S}_i$, $X_{ij}$, $Y_{i}$, $Y_{ij}$ and $Z_{i}$, $i,j\in\mathcal{N}$. Note that implementation of~\eqref{eq:poly-q-stabilizability-analysis-gain} and $K$ in~\ref{item:thm-stab-synth-1} will typically require the inverses $S(p_{k+1})$ and $(S(p_{k+1})+BB^\top)^{-1}$ or $X_i(\pi)^{-1}=(\sum_{j\in\mathcal{N}}\xi_j(\pi)X_{ij})^{-1}$ to be computed on-line. These computationally-expensive on-line computations can be avoided by using~\ref{item:thm-stab-synth-3} or by setting $X_{ij}=\bar{X}_i\in\mathbb{R}^{n_x\times n_x}$, $i,j\in\mathcal{N}$. Note that this was not an option in Theorem~\ref{thm:poly-q-stabilizability-analysis}, since the controller variables do not explicitly appear in~\eqref{eq:poly-q-stabilizability-analysis}. The equivalence $\ref{item:thm-stab-synth-2}\Leftrightarrow\ref{item:thm-stab-synth-3}$ for strictly polytopic systems follows from Lemma~\ref{lem:equiv-pqs}. \ref{item:thm-stab-synth-3} shows that eliminating the slack variables by setting them to $\bar{S}_i$ in~\cite[Theorem 4]{Daafouz2001} is not conservative. 

\section{Numerical example}\label{sec:case-study}
We now demonstrate our results by means of a numerical example. To this end, we consider~\eqref{eq:system} with
\begin{align}
	A(\pi) &= \begin{bmatrix}\begin{smallmatrix}
		\frac{4}{5} & -\frac{1}{4} & 0 & 1\\
		1 & 0 & 0 & 0\\
		0 & 0 & \frac{1}{5} & \frac{3}{100}\\
		0 & 0 & 1 & 0
	\end{smallmatrix}\end{bmatrix}+\pi\begin{bmatrix}\begin{smallmatrix}
		0\\
		0\\
		1\\
		0
	\end{smallmatrix}\end{bmatrix}\begin{bmatrix}\begin{smallmatrix}
		\frac{4}{5}\\
		-\frac{1}{4}\\
		-\frac{1}{5}\\
		-\frac{3}{100}
	\end{smallmatrix}\end{bmatrix}^{\mathclap{\top}},~B=\begin{bmatrix}\begin{smallmatrix}
		1\\
		0\\
		1\\
		0
	\end{smallmatrix}\end{bmatrix}\text{ and }\nonumber\\
	C &= \begin{bmatrix}\begin{smallmatrix}
		1 & 0 & 0 & 0
	\end{smallmatrix}\end{bmatrix},\quad \mathbb{P}_{\gamma}\coloneqq \{p\mid |p|\leqslant \gamma\}, \gamma\in\mathbb{R}_{>0}.\label{eq:matrices}
\end{align}
Using our results, we aim to find the largest $\gamma^\star\in\mathbb{R}_{>0}$ such that the system~\eqref{eq:system} with $p_k\in\mathbb{P}_{\gamma^\star}$, $k\in\mathbb{N}$, is poly-QS/poly-Q stabilizable/poly-Q detectable for all $k\in\mathbb{N}$ and $\bm{p}\in\mathcal{P}_{\gamma^\star}\coloneqq \{\{p_k\}_{k\in\mathbb{N}}\mid p_k\in\mathbb{P}_{\gamma^\star}, k\in\mathbb{N}\}$. Based on the strictly polytopic representation $A(\pi)=((\gamma-\pi)/(2\gamma))A(-\gamma) + ((\gamma+\pi)/(2\gamma))A(\gamma)$, all results in the prequel are applied, using YALMIP toolbox~\cite{Lofberg2004} for MATLAB with external solver MOSEK~\cite{Mosek2019} to solve the LMI-based conditions.

\subsection{Comparison of different conditions}
As summarized in Table~\ref{tab:case-study}, the system is poly-QS for all $|\gamma|\leqslant \gamma^\star= 0.6840$ and the same $\gamma^\star$ is found using all of the conditions in Lemma~\ref{lem:equiv-pqs}, which supports their equivalence. Regarding poly-Q detectability, Theorem~\ref{thm:poly-q-detectability-analysis} and the conditions in Remark~\ref{rem:heemels2010} are equivalent so we expect the obtained $\gamma^\star$ to be equal, however, as seen in Table~\ref{tab:case-study} we compute slightly different $\gamma^\star$'s due to numerics. Next, we test poly-Q stabilizability. Firstly, we apply~\cite[Theorem 4]{Daafouz2001} to compare the result with~\ref{item:thm-stab-synth-3}, which, since both $\gamma^\star$'s are equal, supports their equivalence. Finally, we see that there is no significant difference between the $\gamma^\star$'s obtained through Theorem~\ref{thm:poly-q-stabilizability-analysis} and~\ref{item:thm-stab-synth-4} and the one obtained using~\ref{item:thm-stab-synth-3} and~\cite[Theorem 4]{Daafouz2001}. Looking at this, it appears as if the $p_{k+1}$-dependence in~\eqref{eq:poly-q-stabilizability-analysis-gain} does not translate to less conservative results, which may indicate that perhaps the polytopic structure in~\ref{item:thm-stab-synth-2} is (similar to the observer setting) already rich enough to be non-conservative for poly-Q stabilizability. However, other potential sources of conservatism in Theorem~\ref{thm:poly-q-stabilizability-analysis} prevent us from drawing any conclusions without further research.
\begin{table}[!tb]
\centering
\caption{Stability/stabilizability/detectability bounds $\gamma^\star$}
\label{tab:case-study}%
\begin{tabular}{@{}cc|cc|cc@{}}
\hline
\multicolumn{2}{c}{Poly-QS} & \multicolumn{2}{c}{Poly-Q detectability} & \multicolumn{2}{c}{Poly-Q stabilizability}\\
\hline
Conditions & $\gamma^\star$ & Conditions & $\gamma^\star$ & Conditions & $\gamma^\star$\\
\hline
\ref{item:lem-equiv-pqs-2} & $0.6840$ & Theorem~\ref{thm:poly-q-detectability-analysis} & $4.1917$ & Theorem~\ref{thm:poly-q-stabilizability-analysis} & $1.2198$\\
\ref{item:lem-equiv-pqs-3} & $0.6840$ & Remark~\ref{rem:heemels2010} & $4.2009$ & \ref{item:thm-stab-synth-3} & $1.2202$\\
\ref{item:lem-equiv-pqs-4} & $0.6840$ & & & \ref{item:thm-stab-synth-4} & $1.2201$\\
 & & & & \cite[Thm 4]{Daafouz2001} & $1.2202$
\end{tabular}
\vspace*{-0.5cm}
\end{table}
\begin{table}[!tb]
\centering
\caption{Average computation times}
\label{tab:case-study-comp-times}%
\begin{tabular}{@{}cc|cc|cc@{}}
\hline
\multicolumn{2}{c}{Poly-QS} & \multicolumn{2}{c}{Poly-Q detectability} & \multicolumn{2}{c}{Poly-Q stabilizability}\\
\hline
Conditions & Time [s] & Conditions & Time [s] & Conditions & Time [s]\\
\hline
\ref{item:lem-equiv-pqs-2} & $1.5980$ & Theorem~\ref{thm:poly-q-detectability-analysis} & $1.1293$ & Theorem~\ref{thm:poly-q-stabilizability-analysis} & $2.5588$\\
\ref{item:lem-equiv-pqs-3} & $1.5758$ & Remark~\ref{rem:heemels2010} & $2.0439$ & \ref{item:thm-stab-synth-3} & $2.0011$\\
\ref{item:lem-equiv-pqs-4} & $1.0059$ & & & \ref{item:thm-stab-synth-4} & $8.4978$\\
 & & & & \cite[Thm 4]{Daafouz2001} & $4.3250$
\end{tabular}
\vspace*{-0.5cm}
\end{table}

\subsection{Computational complexity}\label{sec:comp-compl}
\begin{table*}[!t]
\centering
\caption{Number of scalar decision variables \#DV}
\label{tab:dec-vars}%
\begin{tabular}{@{}cc|cc|cc@{}}
\hline
\multicolumn{2}{c}{Poly-QS} & \multicolumn{2}{c}{Poly-Q detectability} & \multicolumn{2}{c}{Poly-Q stabilizability}\\
\hline
Conditions & \#DV & Conditions & \#DV & Conditions & \#DV\\
\hline
\ref{item:lem-equiv-pqs-2} & $N(n_x^2+n_x(n_x+1)/2)$ & Theorem~\ref{thm:poly-q-detectability-analysis} & $Nn_x(n_x+1)/2$ & Theorem~\ref{thm:poly-q-stabilizability-analysis} & $N(n_x^2+n_x(n_x+1)/2)$\\
\ref{item:lem-equiv-pqs-3} & $N(n_x^2+n_x(n_x+1)/2)$ & Remark~\ref{rem:heemels2010} & $N(n_x^2 + n_x(n_x+1)/2 + n_xn_y)$ & \ref{item:thm-stab-synth-3} & $N(n_x(n_x+1)/2+n_xn_u)$\\
\ref{item:lem-equiv-pqs-4} & $Nn_x(n_x+1)/2$ & & & \ref{item:thm-stab-synth-4} & $N(N(n_x^2 + n_xn_u)+n_x^2+n_x(n_x+1)/2)$\\
 & & & & \cite[Thm 4]{Daafouz2001} & $N(n_x^2+n_x(n_x+1)/2+n_xn_u)$
\end{tabular}\\\vspace*{.1cm}
\hrulefill\vspace*{-0.5cm}
\end{table*}
In general, LMIs can be solved with complexity that is proportional to the number of decision variables cubed (i.e., \#DV$^3$), see, e.g.,~\cite{Loquen2010}. Thereby, it is expected that the reduced number of decision variables in, e.g.,~\ref{item:lem-equiv-pqs-4}, Theorem~\ref{thm:poly-q-detectability-analysis} and~\ref{item:thm-stab-synth-3} compared to similar existing conditions translates also to lower computational cost. For small problems, these expected differences in complexity may be overshadowed by computational overhead, therefore, we also apply our results to a somewhat larger system 
\begin{equation*}
    \begin{aligned}
        x_{k+1} &= \operatorname{diag}\{A(p^1_{k}),A(p^2_{k}),A(p^3_{k})\}x_k +\operatorname{diag}\{B,B,B\}u_k,\\
        y_k &= \operatorname{diag}\{C,C,C\}x_k,
    \end{aligned}
\end{equation*}
where $p_k=(p^1_k,p^2_k,p^3_k)\in\mathbb{P}_{\gamma^\star}^3$, $k\in\mathbb{N}$, and $A$, $B$ and $C$ as in~\eqref{eq:matrices}. This system contains $n_x=12$ states, $n_u=3$ inputs and $n_y=3$ outputs as well as $3$ unknown parameters that lie in $\mathbb{P}^3_{\gamma^\star}$. Since this system consists of $3$ decoupled versions of~\eqref{eq:system} with~\eqref{eq:matrices}, its stability, stabilizability and detectability properties remain the same. Hence, we use the values for $\gamma^\star$ that we computed for the system~\eqref{eq:system} with~\eqref{eq:matrices}, which are shown in Table~\ref{tab:case-study}. The recorded average computation times are shown in Table~\ref{tab:case-study-comp-times} as well as formulae for the number of decision variables in the LMIs in Table~\ref{tab:dec-vars}. As expected, we observe that~\ref{item:lem-equiv-pqs-4} requires significantly less time to solve compared to~\ref{item:lem-equiv-pqs-2} and~\ref{item:lem-equiv-pqs-3} since~\ref{item:lem-equiv-pqs-4} does not contain the additional $N$ $n_x$-by-$n_x$ slack variables. Similarly, the conditions in Theorem~\ref{thm:poly-q-detectability-analysis} can be solved faster than those in Remark~\ref{rem:heemels2010} since Remark~\ref{rem:heemels2010} contains similar slack variables as well as the $N$ $n_x$-by-$n_y$ matrices $Y_i$ related to the synthesized observer gains. Finally, we find that we can solve~\ref{item:thm-stab-synth-3} significantly faster than Theorem~\ref{thm:poly-q-detectability-analysis} due to the absence of slack variables. Moreover, the conditions in~\cite[Theorem 4]{Daafouz2001} were also found to be more complex since they require the computation of the controller gains and~\ref{item:thm-stab-synth-4}, although possibly less conservative, contains $N^2$ $n_x$-by-$n_x$ slack variables, which renders it by far the most computationally expensive poly-Q stabilizability condition. These computational benefits are even higher for systems with larger state, input and/or output dimensions, more parameters and parameter sets with more vertices $N$, as also following from Table~\ref{tab:dec-vars}.

\section{Conclusions}\label{sec:conclusions}
In this note, we studied stability, detectability and stabilizability for polytopic LPV systems. By generalizing the Lyapunov-based LMI test for detectability in the LTI setting, we have developed a similar LMI-based test for poly-Q detectability that turned out to be necessary and sufficient. For stabilizability of polytopic systems, we obtained a parameter-dependent necessary and sufficient condition as well as an alternative LMI-based sufficient condition that is no longer parameter dependent and, thereby, can be systematically solved. For all of these results, we discussed their relation to their LTI counterparts and existing synthesis conditions and their computational advantages which were also illustrated by a numerical example.

\bibliographystyle{IEEEtran}
\bibliography{phd-bibtex}

\section*{Appendix}
\subsection{Existing results}
\begin{lemma}[{\cite[Lemma 4]{Iwasaki1995}}]\label{lem:proj-lem}
	Given matrices $\Psi\in\mathbb{S}^{n}$, $\Gamma\in\mathbb{R}^{m\times n}$ and $\Omega\in\mathbb{R}^{p\times n}$ with $\operatorname{rank} \Omega=p<n$. There exists a matrix $\Lambda\in\mathbb{R}^{m\times p}$ such that $\Psi + \Gamma^\top\Lambda\Omega + \Omega^\top\Lambda^\top\Gamma\succ 0$, if and only if $\Gamma_{\perp}^\top\Psi\Gamma_{\perp}\succ 0\text{ and }\Omega_\perp^\top\Psi\Omega_{\perp}\succ 0$.
\end{lemma}

\begin{lemma}[{\cite{Ishihara2017}}]\label{lem:finsler}
	For continuous functions $\Psi\colon\mathbb{P}\rightarrow\mathbb{S}^{n}$ and $\Gamma\colon\mathbb{P}\rightarrow\mathbb{R}^{m\times n}$ on a compact set $\mathbb{P}\subset\mathbb{R}^{p}$, it holds that $\Gamma_{\perp}^\top(\pi)\Psi(\pi)\Gamma_{\perp}(\pi)\succ 0$ for all $\pi\in\mathbb{P}$, if and only if $\Psi(\pi) +\mu \Gamma^\top(\pi)\Gamma(\pi)\succ 0$ for all $\pi\in\mathbb{P}$, for some $\mu\in\mathbb{R}$.
\end{lemma}

\begin{lemma}[{\cite{Zhou1988}}]\label{lem:youngs}
	For matrices $X,Y\in\mathbb{R}^{m\times n}$ and $S\in\mathbb{S}^{m}_{\succ 0}$, it holds that $X^\top Y+Y^\top X\preccurlyeq X^\top S^{-1}X + Y^\top SY$.
\end{lemma}

\subsection{Proofs}
\begin{proof}[Proof of Lemma~\ref{lem:equiv-pqs}]
	It is shown in~\cite{Daafouz2001,Pandey2018} that $\ref{item:lem-equiv-pqs-2}\Rightarrow\ref{item:lem-equiv-pqs-1}$, $\ref{item:lem-equiv-pqs-3}\Rightarrow\ref{item:lem-equiv-pqs-1}$ and, for strictly polytopic systems, $\ref{item:lem-equiv-pqs-1}\Leftrightarrow\ref{item:lem-equiv-pqs-2}\Leftrightarrow\ref{item:lem-equiv-pqs-3}$. Moreover, it is also shown in~\cite{Daafouz2001,Heemels2010} that the matrices $X_i$, $i\in\mathcal{N}$, satisfying~\eqref{eq:lem-equiv-pqs-2} and~\eqref{eq:lem-equiv-pqs-3} are non-singular. We prove $\ref{item:lem-equiv-pqs-2}\Leftrightarrow\ref{item:lem-equiv-pqs-4}$. 
	
	{\bf$\bm{{\ref{item:lem-equiv-pqs-4}\Rightarrow\ref{item:lem-equiv-pqs-2}}}$:} \ref{item:lem-equiv-pqs-4} implies~\ref{item:lem-equiv-pqs-2} with $X_i=\bar{S}_i$. 
	
	{\bf$\bm{{\ref{item:lem-equiv-pqs-2}\Rightarrow\ref{item:lem-equiv-pqs-4}}}$:} If~\ref{item:lem-equiv-pqs-2} holds, then $\bar{S}_i\succ 0$ and
	\begin{equation}
		\begin{bmatrix}\begin{smallmatrix}
			A_i^\top\\
			-I
		\end{smallmatrix}\end{bmatrix}^\top\begin{bmatrix}\begin{smallmatrix}
			X_i+X_i^\top-\bar{S}_i & \star\\
			A_iX_i & \bar{S}_j
		\end{smallmatrix}\end{bmatrix}\begin{bmatrix}\begin{smallmatrix}
			A_i^\top\\
			-I
		\end{smallmatrix}\end{bmatrix}=\bar{S}_j-A_i\bar{S}_iA_i^\top\succ 0,
	\end{equation}
	for all $i,j\in\mathcal{N}$. By Schur complement and congruence transformation with $\operatorname{diag}\{\bar{S}_i,I\}$, $i\in\mathcal{N}$,~\ref{item:lem-equiv-pqs-2} holds.
\end{proof}

\begin{proof}[Proof of Theorem~\ref{thm:poly-q-detectability-analysis}]
	{\bf Sufficiency:} Suppose that there exist $\bar{P}_i\in\mathbb{S}_{\succ 0}$ and $\bar{S}_i\coloneqq\bar{P}_i^{-1}$, $i\in\mathcal{N}$, for which~\eqref{eq:poly-q-detectability-analysis} holds. Let $L_i \coloneqq -A_i(\bar{P}_i+C^\top C)^{-1}C^\top$, $i\in\mathcal{N}$, such that $A_i+L_iC = A_i(I-(\bar{P}_i+C^\top C)^{-1}C^\top C)=A_i(\bar{P}_i+C^\top C)^{-1}\bar{P}_i$, $i,j\in\mathcal{N}$. We denote $Q_{ij}\coloneqq \bar{S}_j-(A_i+L_iC)\bar{S}_i(A_i+L_iC)^\top$, $i,j\in\mathcal{N}$, and will now show that $Q_{ij}\succ 0$, $i,j\in\mathcal{N}$. It holds that $\bar{P}_i\preccurlyeq \bar{P}_i+C^\top C$, $i\in\mathcal{N}$, which, since $Q_{ij} = \bar{S}_j-A_i(\bar{P}_i+C^\top C)^{-1}\bar{P}_i(\bar{P}_i+C^\top C)^{-1}A_i^\top$, $i,j\in\mathcal{N}$, implies that
    \begin{equation}
            Q_{ij}\succcurlyeq \bar{S}_j-A_i(\bar{P}_i+C^\top C)^{-1}A_i^\top,\text{ for all }i,j\in\mathcal{N}.
        \label{eq:interm7}
    \end{equation} 
    Applying Schur complement to the right-hand side~\eqref{eq:interm7} twice shows that $Q_{ij}\succ 0$, for all $i,j\in\mathcal{N}$, since $\bar{P}_i+C^\top C-A_i^\top \bar{P}_jA_i\stackrel{\eqref{eq:poly-q-detectability-analysis}}{\succ} 0$, $i,j\in\mathcal{N}$. Using Schur complement, $Q_{ij}\succ 0$ for all $i,j\in\mathcal{N}$ is equivalent to~\ref{item:lem-equiv-pqs-4} with $A_i\leftarrow A_i+L_iC$, which implies that~\eqref{eq:error-system} with $L$ in~\eqref{eq:poly-q-detectability-analysis-gain} is poly-QS and $V$~\eqref{eq:poly-lyap} is a poly-QLF. Note that we did not use strict polytopicity anywhere, hence, sufficiency holds for the full class of polytopic systems~\eqref{eq:system}.
	
	{\bf Necessity for strictly polytopic systems:} Suppose that the system~\eqref{eq:system} is poly-Q detectable, i.e., there exist $L\colon\mathbb{N}\times\mathcal{P}\rightarrow\mathbb{R}^{n_x\times n_y}$ and $\bar{P}_i\in\mathbb{S}^{n_x}$, $i\in\mathcal{N}$, such that $V$ as in~\eqref{eq:poly-lyap} is a poly-QLF for~\eqref{eq:error-system}. From~\eqref{eq:lyap-desc}, using the Schur complement,
	\begin{equation}
		\begin{bmatrix}\begin{smallmatrix}
			P(p_k) & A^\top(p_k)\\
			\star & S(p_{k+1})
		\end{smallmatrix}\end{bmatrix} + \operatorname{He}\left(\begin{bmatrix}\begin{smallmatrix}
			C^\top\\
			0
		\end{smallmatrix}\end{bmatrix}L^\top(k,\bm{p})\begin{bmatrix}\begin{smallmatrix}
			0 & I
		\end{smallmatrix}\end{bmatrix}\right)\succ 0,
		\label{eq:interm1}
	\end{equation}
	where $S(\pi)=P^{-1}(\pi)\succ 0$, $\pi\in\mathbb{P}$, $k\in\mathbb{N}$ and $\bm{p}\in\mathcal{P}$. By pointwise application of Lemma~\ref{lem:proj-lem} for all $k\in\mathbb{N}$ and $\bm{p}\in\mathcal{P}$, there exists $L$ satisfying~\eqref{eq:interm1} if and only if, for all $\pi_+,\pi\in\mathbb{P}$,
	\begin{equation}
		P(\pi)\succ 0\text{ and }\begin{bmatrix}\begin{smallmatrix}
			C & 0
		\end{smallmatrix}\end{bmatrix}_{\perp}^\top\begin{bmatrix}\begin{smallmatrix}
			P(\pi) & A^\top(\pi)\\
			\star & S(\pi_+)
		\end{smallmatrix}\end{bmatrix}\begin{bmatrix}\begin{smallmatrix}
			C & 0
		\end{smallmatrix}\end{bmatrix}_{\perp}\succ 0.
		\label{eq:Cperp}
	\end{equation}
	Using the continuity of $P$, $S$ and $A$, the second condition in~\eqref{eq:Cperp}, by pointwise application of Lemma~\ref{lem:finsler} for all $\pi_+,\pi\in\mathbb{P}$, holds if and only if there exists $\mu\in\mathbb{R}$ such that $P(\pi) - A^\top(\pi)P(\pi_+)A(\pi) + \mu C^\top C \succ 0$ for all $\pi_+,\pi\in\mathbb{P}$. It follows, due to the strict polytopicity of~\eqref{eq:system}, that $\bar{P}_i\succ 0$ and $\bar{P}_i-A_i^\top \bar{P}_jA_i + \mu C^\top C\succ 0$ for all $i,j\in\mathcal{N}$. Since $C^\top C\succcurlyeq 0$,~\eqref{eq:poly-q-detectability-analysis} holds with $\bar{\mu}>\max\{0,\mu\}$ and $\bar{P}_i\leftarrow (1/\bar{\mu})\bar{P}_i$.
\end{proof}

\begin{proof}[Proof of Theorem~\ref{thm:poly-quadratic-stabilizability-nec-suff}]
    {\bf Sufficiency:} Suppose there exist matrices $\bar{S}_i\in\mathbb{S}^{n_x}_{\succ 0}$ and $\bar{P}_i=\bar{S}_i^{-1}$, $i\in\mathcal{N}$, such that $S(\pi)=P^{-1}(\pi)$ with $P$ as in~\eqref{eq:poly-lyap} satisfies~\eqref{eq:poly-quadratic-stabilizability-nec-suff} for all $\pi_+,\pi\in\mathbb{P}$. It follows, from $\bar{S}_i\succ 0$, $i\in\mathcal{N}$, that $S(\pi)\succ 0$ for all $\pi\in\mathbb{P}$. Let $K$ be as in~\eqref{eq:poly-q-stabilizability-analysis-gain} such that $A(p_k)+BK(k,\bm{p}))=(I-BB^\top(S(p_{k+1})+BB^\top)^{-1})A(p_k)=S(p_{k+1})(S(p_{k+1})+BB^\top)^{-1}A(p_k)$, $k\in\mathbb{N}$, $\bm{p}\in\mathcal{P}$. We denote $Q(k,\bm{p}) \coloneqq P(p_k)-(A(p_k)+BK(k,\bm{p}))^\top P(p_{k+1})(A(p_k)+BK(k,\bm{p}))$, $k\in\mathbb{N}$, $\bm{p}\in\mathcal{P}$, and will now show that $Q(k,\bm{p})\succ 0$ for all $k\in\mathbb{N}$ and $\bm{p}\in\mathcal{P}$. It holds that $-S(\pi)\succcurlyeq -(S(\pi)+BB^\top)$ for all $\pi\in\mathbb{P}$, which, by expressing $Q(k,\bm{p})=P(p_k)-A^\top(p_k)(S(p_{k+1})+BB^\top)^{-1}S(p_{k+1})(S(p_{k+1})+BB^\top)^{-1}A(p_k)$, $k\in\mathbb{N}$, $\bm{p}\in\mathcal{P}$, implies that, for all $k\in\mathbb{N}$ and $\bm{p}\in\mathcal{P}$,
    \begin{equation}
        Q(k,\bm{p})\succcurlyeq P(p_k)-A^\top(p_k)(S(p_{k+1})+BB^\top)^{-1}A(p_k).
        \label{eq:interm8}
    \end{equation}
    Applying Schur complement to the right-hand side of~\eqref{eq:interm8} twice shows that $Q(k,\bm{p})\succ 0$ for all $k\in\mathbb{N}$ and $\bm{p}\in\mathcal{P}$, since $S(p_{k+1})+BB^\top -A(p_k)S(p_k)A^\top(p_k)\stackrel{\eqref{eq:poly-quadratic-stabilizability-nec-suff}}{\succ} 0$ for all $\bm{p}\in\mathcal{P}$ and $k\in\mathbb{N}$. We conclude, using also the fact that $P$ is uniformly bounded on $\mathbb{P}$, that $V$ in~\eqref{eq:poly-lyap} is a poly-QLF for~\eqref{eq:closed-loop} and, thus, $K$ in~\eqref{eq:poly-q-stabilizability-analysis-gain} renders~\eqref{eq:closed-loop} poly-QS.

    {\bf Necessity:} Suppose that system~\eqref{eq:system} is poly-Q stabilizable, i.e., there exist $K\colon\mathbb{N}\times\mathcal{P}\rightarrow\mathbb{R}^{n_u\times n_x}$ and $\bar{P}_i\in\mathbb{S}^{n_x}$, $i\in\mathcal{N}$, such that $V$ as in~\eqref{eq:poly-lyap} is a poly-QLF for~\eqref{eq:closed-loop}. Then, $\bar{P}_i\succ 0$, $\bar{S}_i=\bar{P}_i^{-1}$, $i\in\mathcal{N}$, and $S(\pi)=P^{-1}(\pi)\succ 0$, $\pi\in\mathbb{P}$. By Schur complement,~\eqref{eq:lyap-desc} yields, for all $\bm{p}\in\mathcal{P}$, $k\in\mathbb{N}$,
	\begin{equation}
		\begin{bmatrix}\begin{smallmatrix}
			P(p_k) & \star\\
			A(p_k) & S(p_{k+1}) 
		\end{smallmatrix}\end{bmatrix} + \operatorname{He}\left(\begin{bmatrix}\begin{smallmatrix}
			0\\
			B
		\end{smallmatrix}\end{bmatrix}K(k,\bm{p})\begin{bmatrix}\begin{smallmatrix}
			I & 0
		\end{smallmatrix}\end{bmatrix}\right)\succ 0.
		\label{eq:interm3}
	\end{equation}
	By pointwise application of Lemma~\ref{lem:proj-lem} for all $k\in\mathbb{N}$ and $\bm{p}\in\mathcal{P}$, there exists $K$ satisfying~\eqref{eq:interm3} if and only if
	\begin{equation}
		S(\pi)\succ 0\text{ and }\begin{bmatrix}\begin{smallmatrix}
			0 & B^\top
		\end{smallmatrix}\end{bmatrix}_{\perp}^\top\begin{bmatrix}\begin{smallmatrix}
			P(\pi) & \star\\
			A(\pi) & S(\pi_+) 
		\end{smallmatrix}\end{bmatrix}\begin{bmatrix}\begin{smallmatrix}
			0 & B^\top
		\end{smallmatrix}\end{bmatrix}_{\perp}\succ 0,
	\end{equation}
	for all $\pi_+,\pi\in\mathbb{P}$. We can complete the proof by applying Lemma~\ref{lem:finsler} pointwise for all $\pi_+,\pi\in\mathbb{P}$ and following the same steps as at the end of the proof of Theorem~\ref{thm:poly-q-detectability-analysis}.
\end{proof}

\begin{proof}[Proof of Theorem~\ref{thm:poly-q-stabilizability-analysis}]
	Suppose that there exist $\bar{P}_i\in\mathbb{S}^{n_x}$ and $X_i\in\mathbb{R}^{n_x\times n_x}$, $i\in\mathcal{N}$, such that~\eqref{eq:poly-q-stabilizability-analysis} holds with $\bar{S}_i\coloneqq \bar{P}^{-1}_i$, $i\in\mathcal{N}$. It follows that $\bar{S}_i\succ 0$, $i\in\mathcal{N}$, and, by congruence transformation with $\operatorname{diag}\{I,\bar{P}_j\}$ and Schur complement, it follows from~\eqref{eq:poly-q-stabilizability-analysis} that
    \begin{equation}
        \begin{bmatrix}\begin{smallmatrix}
			\bar{P}_i & \star & \star\\
			A_i & \operatorname{He}(X_i) +BB^\top & \star\\
			0 & \bar{P}_jX_i & \bar{P}_j
		\end{smallmatrix}\end{bmatrix}\succ 0,\text{ for all }i,j\in\mathcal{N}.
        \label{eq:interm9}
    \end{equation}
    Multiplying~\eqref{eq:interm9} by $\xi_i(\pi)\xi_j(\pi_+)$ and summing over all $i,j\in\mathcal{N}$, we obtain
	\begin{equation}
		\begin{bmatrix}\begin{smallmatrix}
			P(\pi) & \star & \star\\
			A(\pi) & \operatorname{He}(X(\pi)) +BB^\top & \star\\
			0 & P(\pi_+)X(\pi) & P(\pi_+)
		\end{smallmatrix}\end{bmatrix}\succ 0,\text{ for all }\pi_+,\pi\in\mathbb{P},\label{eq:interm4}
	\end{equation}
	with $P(\pi)=\sum_{i\in\mathcal{N}}\xi_i(\pi)\bar{P}_i$, as in~\eqref{eq:poly-lyap}, and $X(\pi)=\sum_{i\in\mathcal{N}}\xi_i(\pi)X_i$, $\pi\in\mathbb{P}$. Since $\bar{P}_i\succ 0$, $i\in\mathcal{N}$, we have $P(\pi)\succ 0$, $\pi\in\mathbb{P}$. From~\eqref{eq:interm4}, we obtain $X(\pi)+X^\top(\pi)-X^\top(\pi)P(\pi_+)X(\pi)-A(\pi)S(\pi)A^\top(\pi)+BB^\top\succ 0$ for all $\pi_+,\pi\in\mathbb{P}$. By pointwise application of Lemma~\ref{lem:youngs}, we conclude that $S$ satisfies~\eqref{eq:poly-quadratic-stabilizability-nec-suff} for all $\pi_+,\pi\in\mathbb{P}$. Thus, by Theorem~\ref{thm:poly-quadratic-stabilizability-nec-suff}, $K$ as in~\eqref{eq:poly-q-stabilizability-analysis-gain} with $S(\pi)=P^{-1}(\pi)$, $\pi\in\mathbb{P}$, and $P$ as in~\eqref{eq:poly-lyap} renders~\eqref{eq:closed-loop} poly-QS with poly-QLF~\eqref{eq:poly-lyap}.
\end{proof}

\addtolength{\textheight}{-17.5cm}   

\begin{proof}[Proof of Theorem~\ref{thm:stab-synth}]
	{\bf$\bm{{\ref{item:thm-stab-synth-2}\Rightarrow\ref{item:thm-stab-synth-1}}}$:} If~\ref{item:thm-stab-synth-2} holds, then~\ref{item:thm-stab-synth-1} holds with $X_{ij}\leftarrow I$ and $\bar{K}_{ij}\leftarrow \bar{K}_i$, $i,j\in\mathcal{N}$. 
		
	{\bf$\bm{{\ref{item:thm-stab-synth-3}\Rightarrow\ref{item:thm-stab-synth-2}}}$:} If~\ref{item:thm-stab-synth-3} holds, then $\bar{S}_i\succ 0$ and
	\begin{equation}
		\begin{bmatrix}\begin{smallmatrix}
			\bar{S}_i & \star\\
			(A_i+B\bar{K}_i)\bar{S}_i & \bar{S}_j
		\end{smallmatrix}\end{bmatrix}\succ 0,\text{ for all }i,j\in\mathcal{N},
	\end{equation}
	which, by~\ref{item:lem-equiv-pqs-4} with $A_i\leftarrow A_i+B\bar{K}_i$, implies~\ref{item:thm-stab-synth-2} with the same $K_i$, $i\in\mathcal{N}$, and with~\eqref{eq:poly-lyap} being a poly-QLF for~\eqref{eq:closed-loop}.
	
	{\bf $\bm{{\ref{item:thm-stab-synth-2}\Rightarrow\ref{item:thm-stab-synth-3}}}$ for strictly polytopic systems:} Suppose that $\ref{item:thm-stab-synth-2}$ holds, then, due to strict polytopicity,~\ref{item:thm-stab-synth-3} (with non-singular $X_i$, $i\in\mathcal{N}$) follows from~\ref{item:lem-equiv-pqs-4}.
	
	{\bf$\bm{{\ref{item:thm-stab-synth-4}\Rightarrow\ref{item:thm-stab-synth-1}}}$:} If~\ref{item:thm-stab-synth-4} holds, then $\bar{P}_i=\bar{S}_i^{-1}\succ 0$, $i\in\mathcal{N}$. Congruence transformation with $\operatorname{diag}\{I,I,\bar{P}_j\}$, multiplying by $\xi_j(\pi_+)$ and summing over all $j\in\mathcal{N}$ yield
	\begin{equation}
		\begin{bmatrix}\begin{smallmatrix}
			\operatorname{He}(X_i(\pi_+))-\bar{S}_i & \star & \star\\
			A_iX_i(\pi_+)+BY_i(\pi_+) & \operatorname{He}(Z_i) & \star\\
			0 & P(\pi_+)Z_i & P(\pi_+)
		\end{smallmatrix}\end{bmatrix}\succ 0,\label{eq:interm6}
	\end{equation}
	with $\left[\begin{smallarray}{@{}ccc@{}} X_i(\pi_+) & Y_i(\pi_+) & P(\pi_+)\end{smallarray}\right]=\sum_{j\in\mathcal{N}}\xi_j(\pi_+)\left[\begin{smallarray}{@{}ccc@{}}X_{ij} & Y_{ij} & \bar{P}_i\end{smallarray}\right]$, for all $\pi_+\in\mathbb{P}$, $i\in\mathcal{N}$. From~\eqref{eq:interm6}, $X_i(\pi)$, $i\in\mathcal{N}$, are non-singular for all $\pi\in\mathbb{P}$. Applying Lemma~\ref{lem:youngs} to~\eqref{eq:interm6} followed by a congruence transformation with $\operatorname{diag}\{X_i^{-1}(\pi_+),I,I\}$ yields
	\begin{equation*}
		\begin{bmatrix}\begin{smallmatrix}
			\bar{P}_i & \star & \star\\
			A_i+ B\bar{K}_i(\pi_+) & \operatorname{He}(Z_i) & \star\\
            0 & P(\pi_+)Z_i & P(\pi_+)
		\end{smallmatrix}\end{bmatrix}\succ 0,\text{ for all }i\in\mathcal{N}, \pi_+\in\mathbb{P},
	\end{equation*}
	where $\bar{K}_i(\pi) = Y_i(\pi)X_i^{-1}(\pi)$ and $S(\pi)=P^{-1}(\pi)$, $\pi\in\mathbb{P}$, $i\in\mathcal{N}$. Applying the Schur complement and Lemma~\ref{lem:youngs}, it follows that
    \begin{equation}
        \begin{bmatrix}\begin{smallmatrix}
            \bar{P}_i & \star \\
            A_i+B\bar{K}_i(\pi_+) & S(\pi_+)
        \end{smallmatrix}\end{bmatrix}\succ 0,\text{ for all }i\in\mathcal{N},\pi_+\in\mathbb{P}.
    \end{equation}
    Next, we multiply by $\xi_i(\pi)$, sum over all $i\in\mathcal{N}$ and apply Schur complement to obtain $P(p_k) - (A(p_k)+BK(k,\bm{p}))^{\top}P(p_{k+1})(A(p_k)+BK(k,\bm{p}))\succ 0$ for all $k\in\mathbb{N}$, $\bm{p}\in\mathcal{P}$, where $K(k,\bm{p})=\sum_{i\in\mathcal{N}}\xi_i(p_k)\bar{K}_i(p_{k+1})$, $k\in\mathbb{N}$, $\bm{p}\in\mathcal{P}$. Since $P$, $A$ and $K$ are polytopic, $V$ as in~\eqref{eq:poly-lyap} satisfies~\eqref{eq:lyap-bounds}-\eqref{eq:lyap-desc} for~\eqref{eq:closed-loop} with some $a_i\in\mathbb{R}_{>0}$, $i\in\mathbb{N}_{[1,3]}$.
\end{proof}

\end{document}